\documentclass[11pt]{compositio}

\usepackage[T1]{fontenc}
\usepackage[utf8]{inputenc}
\usepackage{latexsym}
\usepackage{amsfonts}
\usepackage{amstext}
\usepackage{amsmath}
\usepackage{graphicx}
\usepackage{hyperref}
\usepackage[hyphenbreaks]{breakurl}
\usepackage{enumitem}
\usepackage{mathtools}
\usepackage{amssymb,euscript,multicol,graphicx}
\usepackage[ec]{aeguill}       
\usepackage{delarray, fancyhdr}
\usepackage[active]{srcltx}
\usepackage[all,dvips]{xy}
\input xy
\xyoption{all}

\def\Q{{\mathbb Q}}

\def\Z{{\mathbb Z}}

\def\Mc{{\mathcal M}}

\def\Oc{{\mathcal O}}
\def\Pc{{\mathcal P}}

\def\ago{{\mathfrak a}}

\def\mgo{{\mathfrak m}}

\def\eps{\varepsilon}

\def\ov{\overline}
\newcommand{\cartesien}{\ar@{}[dr]|{\square}}

\newcommand{\double}{\ar@<2pt>[r] \ar@<-2pt>[r]}

\def\Com{\hbox{\rm Com}\,}

\DeclareMathOperator{\edim}{edim}

\DeclareMathOperator{\Frob}{Frob}

\def\Gal{\hbox{\rm Gal}\,}
\def\GL{\hbox{\rm GL}\,}

\def\Ker{\hbox{\rm Ker}\,}

\DeclareMathOperator{\Spec}{Spec}

\DeclareMathOperator{\Supp}{Supp}

\newcommand{\details}[1]{}

\newtheorem{cor}[subsection]{Corollary}

\newtheorem{prop}[subsection]{Proposition}

\newtheorem{thm}[subsection]{Theorem}

\newtheorem{defi}[subsection]{Definition}

\newtheorem{lem}[subsection]{Lemma}

\newtheorem{remarque}[subsection]{Remark}

\title{Proof of De Smit's conjecture: a freeness criterion}
\date{\today}

\def\vide{\varnothing}
\def\lto{%
\xymatrix{\ar[r] &}}

\begin{document}

\author{Sylvain Brochard}
\email{sylvain.brochard@umontpellier.fr}
\address{Institut Montpelliérain Alexander Grothendieck, CNRS, Univ. Montpellier
}


  \begin{abstract}
Let $A\to B$ be a morphism of Artin local 
rings with the same embedding dimension. We prove that any $A$-flat 
$B$-module is $B$-flat. This freeness criterion was conjectured by de Smit in 
1997 and improves Diamond's criterion~\cite[Theorem 
2.1]{Diamond_The_Taylor_Wiles}. We also prove that if there is a nonzero 
$A$-flat $B$-module, then $A\to B$ is flat and is a relative complete 
intersection. Then we explain how this result allows one to simplify Wiles's 
proof 
of Fermat's Last Theorem: we do not need the so-called ``Taylor-Wiles systems'' 
any more.
  \end{abstract}
  
  \maketitle

\subsection*{Notations.} If $A$ is a local ring, we 
denote by $\mgo_A$ its maximal ideal and by $\kappa(A)$ its 
residue field. The embedding dimension of $A$, i.e. the minimal number of 
generators of $\mgo_A$, is denoted by $\edim(A)$. A local morphism $A\to B$ of 
Noetherian local rings is a relative complete intersection if the 
ring~$B/\mgo_A B$ is a complete intersection.

\section{Introduction}

A crucial step in the proof of the modularity of semistable elliptic curves 
over~$\Q$ by Taylor and Wiles is to show that certain Hecke algebras are 
complete intersections.
Their method, now commonly referred-to as a ``patching method'', 
heuristically consists of considering the limit of algebras arising from forms 
of different levels.
In their construction, they used in particular the following (fundamental) 
``multiplicity one'' result: the homology of the modular curve is a free module 
over the Hecke algebra (after localization at some maximal ideal).
Diamond reversed the argument in~\cite{Diamond_The_Taylor_Wiles}: ``multiplicity 
one'' becomes a byproduct of the proof rather than an ingredient.
The key input from commutative algebra that allowed Diamond's improvement was 
the freeness criterion~\cite[Thm 2.1]{Diamond_The_Taylor_Wiles}. Its proof 
still relies on a patching argument.

In the hypotheses of Diamond's freeness criterion, there is a condition that 
has to hold for all positive integers $n$. The condition for $n$ implies the 
condition for all smaller integers.
In \cite[Remark 2.2]{Diamond_The_Taylor_Wiles}, Diamond asks whether a 
sufficiently large~$n$ would be enough, and whether there could exist a lower 
bound for such an $n$ depending only on the number $r$ of variables and the 
rank of the module. Around 1997, de Smit conjectured 
that if $A\to B$ is a morphism of Artin local rings with the same embedding 
dimension, then any $A$-flat $B$-module is $B$-flat 
(see Corollary~\ref{cor:desmit} below). In~\cite{Brochard_Mezard_ConjDeSmit} we 
proved de Smit's conjecture when 
the embedding dimension $r$ is $\leq 2$, and with the additional assumption 
that $A\to B$ was flat. Here we prove the following theorem, which in 
particular proves the full conjecture, and
answers Diamond's question: $n=2$ already suffices, for any $r$ and any module. 
So, the freeness criterion will be much easier to apply.

\begin{thm}
    \label{thm_desmit}
Let $\varphi : A\to B$ be a local morphism of Noetherian local rings. Assume 
the following:
\begin{enumerate}[label=(\roman*)]
\item $\edim(B)\leq \edim(A)$
\item There is a nonzero $A$-flat $B$-module $M$.
\item $A$ and $B$ are Artin or $M$ is of finite type as a $B$-module.
\end{enumerate}
Then:
\begin{enumerate}[label=(\arabic*)]
\item The morphism $A\to B$ is flat and $\edim(B)=\edim(A)$.
\item The ring $B/\mgo_AB$ is a complete intersection of dimension zero.
\item $M$ is flat over $B$.
\end{enumerate}
\end{thm}

\begin{cor}[(de Smit's conjecture)]
 \label{cor:desmit}
Let $A\to B$ be a morphism of Artin local rings with the same embedding 
dimension. Then any $A$-flat $B$-module is $B$-flat.
\end{cor}

\begin{cor}[(see Theorem~\ref{thm:global_version})] 
Let $f : X\to S$ be a morphism of locally Noetherian schemes such that for 
any $x\in X$, the tangent spaces satisfy
\(
 \dim T_x \leq \dim T_{f(x)}
\). Then any $S$-flat coherent $\Oc_X$-module is also flat over $X$.
\end{cor}

Since its publication, Diamond's freeness criterion and the related patching 
method have been used quite a lot. For instance, here is a list of papers 
(nonexhaustive, in chronological order) that directly use Diamond's Theorem 
2.1:
\cite{Conrad_Diamond_Taylor},
 \cite{Dickinson_On_The_Modularity},
 \cite{Genestier_Tilouine_Systemes},
 \cite{Bockle_Khare_Modl},
 \cite{Taylor_On_The_Meromorphic},
\cite{Clozel_Harris_Taylor},
 \cite{Dimitrov_On_Iharas_Lemma},
 \cite{Pilloni_Modularite_formes_de},
 \cite{Harris_The_Taylor_Wiles},
\cite{Breuil_Diamond_Formes}.
A variant of this freeness criterion, due to Kisin 
(see \cite[(3.3.1)]{Kisin_Moduli_Of_Finite}), was also used in several major 
results, for instance 
the proof of Serre's modularity conjecture 
(see \cite{Kisin_Moduli_Of_Finite},
\cite{Khare_Wintenberger_Serre2} and
 \cite{Kisin_Modularity_Of_2adic}) 
or the Fontaine--Mazur conjecture for $GL_2$ 
(see \cite{Kisin_The_Fontaine_Mazur}). A more functorial approach of the 
patching method is proposed in~\cite{Emerton_Gee_Savitt}.

The paper is organized as follows: in section~\ref{section:preuve_thm_desmit} 
we prove Theorem~\ref{thm_desmit}. We also give a global version for 
schemes that are not necessarily local (see Theorem~\ref{thm:global_version}). 
In 
section~\ref{section:preuve_flt} we briefly explain how this new freeness 
criterion allows one to simplify 
Wiles's proof of Fermat's Last Theorem and avoid the use of the patching 
method.

  \section{Proof of Theorem \ref{thm_desmit}}
\label{section:preuve_thm_desmit}
 
The following lemma is the heart of the proof.

\begin{lem}
 \label{lem:wtf}
Let $B$ be a ring. Let $(x_1, \dots, x_n)$ and 
$(u_1, \dots, u_n)$ be sequences of elements of $B$ and let us denote by $J_x$ 
and $J_u$ the ideals they generate. Let $M$ be a $B$-module. Assume that:
\begin{enumerate}
 \item We have the inclusion of ideals $J_x \subset J_u$, in other words, if 
$x$ and $u$ denote 
the column matrices $(x_1, \dots, x_n)^T$ and $(u_1, \dots, u_n)^T$, there is a 
square $n\times n$ matrix $W=(c_{ij})$ with entries in $B$ such 
that $x=Wu$. Let $\Delta=\det W$.
\item For any $m_1, \dots, m_n\in M$, if we have the relation
$ \sum_{i=1}^n x_im_i=0$, then $m_i\in J_xM$ for all $i$.
\end{enumerate}
Then, for any $m\in M$, if $\Delta m\in J_x M$, 
then $m\in J_uM$.
\end{lem}
\begin{proof}
We will need some notations for the minors of $W$. 
For any
  integer $\ell$, let~$E_{\ell}$ be the set~$\{1, \dots, 
  \ell\}$ and let $\Pc_{\ell}$ be the set of subsets 
  $I\subset E_n$ such that $|I|={\ell}$. For~$I, J\in 
  \Pc_{\ell}$, let~$\Delta^I_J$ denote the minor of $W$ 
  obtained by deleting the rows (respectively the columns) whose 
  index is in $I$ (respectively~$J$). In other words, 
  $\Delta^I_J=\det (c_{ij})_{\substack{i\in 
      E_n\smallsetminus I\\ j\in E_n\smallsetminus J}}$. In 
  particular, $\Delta_{\vide}^{\vide}=\Delta$ and $\Delta_{E_n}^{E_n}=1$. The 
comatrix is $\Com 
(W)=((-1)^{i+j}\Delta^i_j)$.
  If $I\subset E_n$ and $i\in E_n$, we write $I-i$ for $I\smallsetminus \{i\}$. 
  For $i\neq j\in E_n$, let~$\ov{\eps}(i,j)=1$ if $i<j$ and
  $\ov{\eps}(i,j)=-1$ if $i>j$. For $I\subset E_n$ and
  $i\in I$, let $\eps(i,I)=(-1)^p$,
where~$p$ is the position of $i$ in the  ordered set $E_n\smallsetminus (I-i)$. 
Note that $\eps(i,I)=(-1)^i\prod_{j\in I-i} \ov{\eps}(i,j)$. For any~$i,j\in I$ 
with $i\neq j$, we have the relation
  \begin{equation}
    \label{eq_signes}
    \eps(i,I)\eps(i,I-j)=\ov{\eps}(i,j)
  \end{equation}
Now, for any $0\leq \ell\leq n-1$, if $I\in \Pc_{\ell+1}$ and $i\in I$, the 
expansion of the minor
  $\Delta^{E_{\ell}}_{I-i}$ along the column $i$ is
  \begin{equation}
    \label{eq_dev_mineur}    
   \Delta^{E_{\ell}}_{I-i}=(-1)^{\ell}\eps(i,I)\sum_{k=\ell+1}^n
(-1)^k c_{ki}\Delta_I^{E_{\ell}\cup \{k\}}
\end{equation}
If $i\notin I$, pick an element $j\in I$ and 
consider the expansion along the column $j$ of the 
minor~$\Delta^{E_{\ell}}_{I-j}$ in which we have replaced the 
column $j$ by the column~$i$. Since the latter minor has two identical columns, 
it is zero and we see that for any $i\notin I$,
\begin{equation}
 \label{eq_dev_nul}
  \sum_{k=\ell+1}^n
(-1)^k c_{ki}\Delta_{I}^{E_{\ell}\cup \{k\}}=0
\end{equation}

Now let us come back to our module $M$. Let $m\in M$ be such that $\Delta m\in 
J_x M$. For $0\le {\ell}\le n$, let $(A_{\ell})$ denote the 
following statement.

$(A_{\ell})$: There is a family $(a_I^{{\ell}})_{I\in 
\Pc_{{\ell}-1}}$ of elements of $M$ such that for any $I\in 
\Pc_{\ell}$, the element
\[
g_I=\Delta_I^{E_{\ell}}m
+ \sum_{i\in I} \eps(i,I)u_ia_{I-i}^{{\ell}}
\]
belongs to $J_xM$.

We will prove that $(A_{\ell})$ holds for any $0\le {\ell}\le n$ by 
induction on $\ell$. The statement $(A_0)$ means 
that $\Delta m\in J_x M$ and holds by assumption. Assume that 
$(A_{\ell})$ holds for some $0\leq {\ell}\le n-1$ and let us prove 
the statement $(A_{\ell+1})$.
For any $I\in \Pc_{\ell}$, since $g_I\in J_xM$, 
we can write $g_I=\sum_{k\in E_n}x_k a_I^k$ for some 
elements $a_I^k\in M$.
 Let $I\in 
\Pc_{\ell+1}$. We will compute 
$m'=\sum_{i\in I}\eps(i,I)
u_ig_{I-i}$. For this, we first compute
\begin{align*}
  \sum_{i\in I}\eps(i,I) u_i(g_{I-i}-\Delta_{I-i}^{E_{\ell}}m)
  &= \sum_{i\in I}\eps(i,I) u_i
  \sum_{j\in I-i} 
  \eps(j,I-i)u_j
  a_{I\smallsetminus\{i,j\}}^{{\ell}} 
  \\
  &= \sum_{\substack{i,j\in I \\i\neq j}}
  u_iu_ja_{I\smallsetminus\{i,j\}}^{{\ell}}
  S(i,j)
\end{align*}
where $S(i,j)=\eps(i,I)\eps(j,I-i)$. 
Now the sum vanishes because $S(i,j)=-S(j,i)$. Indeed,
$S(i,j)S(j,i)=\ov{\eps}(i,j) \ov{\eps}(j,i)$ by 
equation~\eqref{eq_signes} above.
Hence,
\begin{equation*}
 m'=\sum_{i\in I}\eps(i,I) 
u_i\Delta_{I-i}^{E_{\ell}}m
\end{equation*}
We expand $\Delta_{I-i}^{E_{\ell}}$ 
along the column $i$:
\begin{align*}
   m' &=\sum_{i\in I}\eps(i,I)u_i(-1)^{\ell}\eps(i,I)
\sum_{k=\ell+1}^n (-1)^kc_{ki} \Delta_I^{E_{\ell}\cup \{k\}}m \\
 &= (-1)^{\ell} \sum_{k=\ell+1}^n (-1)^{k}
\Delta_{I}^{E_{\ell}\cup \{k\}}
\left(\sum_{i\in I}c_{ki}u_i\right)m\\
  &=(-1)^{\ell}\sum_{k=\ell+1}^n 
(-1)^{k}\Delta_{I}^{E_{\ell}\cup \{k\}}
\left(x_k-\sum_{i\notin I}c_{ki}u_i \right)m
 \quad \quad \textrm{(because  $x=Wu$)} \\
&=(-1)^{\ell}\sum_{k=\ell+1}^n 
(-1)^{k}\Delta_{I}^{E_{\ell}\cup \{k\}}x_km \qquad \qquad
\textrm{(using \eqref{eq_dev_nul} for each $i\notin I$)}
\end{align*}
On the other hand, by definition
\begin{equation*}
 m'= \sum_{i\in I}\eps(i,I)u_ig_{I-i}
= \sum_{i\in I}\eps(i,I)u_i \sum_{k\in E_n}x_k a_{I-i}^k
\end{equation*}
Hence,
\[
\sum_{i\in I}\eps(i,I)u_i \sum_{k\in E_n}x_k a_{I-i}^k
- (-1)^{\ell}\sum_{k=\ell+1}^n 
(-1)^{k}\Delta_{I}^{E_{\ell}\cup \{k\}}x_km =0\, .
\]
By our assumption (ii), the coefficient of $x_{\ell+1}$ 
in the above equation belongs to $J_x M$. But this coefficient is
\[
\sum_{i\in I}\eps(i,I)u_i a_{I-i}^{\ell+1}
+\Delta_{I}^{E_{\ell+1}}m\, .
\]
This proves the statement $(A_{\ell+1})$.
In particular, $(A_n)$ holds and hence 
there are elements $a_I^{n}\in M$ for $I\in 
\Pc_{n-1}$  such that the element
\[
g_{E_n}=\Delta_{E_n}^{E_n}m
+ \sum_{i\in E_n} \eps(i,E_n)u_ia_{E_n-i}^n
\]
belongs to $J_xM$. Since $\Delta_{E_n}^{E_n}=1$,
this proves that $m\in J_uM$.
\end{proof}

Before proceeding to the proof of Theorem~\ref{thm_desmit}, let us recall a 
result 
from~\cite{Brochard_Mezard_ConjDeSmit}. It will provide a useful 
characterization for the flatness of $M$ over~$B$.

\begin{defi}[{\cite[3.1]{Brochard_Mezard_ConjDeSmit}}]
 Let $R$ be a local ring and $M$ be an $R$-module. We say that $M$ is weakly 
torsion-free if, for every $\lambda\in R$ and every $m\in M$, the relation 
$\lambda m=0$ implies that $\lambda=0$ or~$m\in \mgo_R M$.
\end{defi}

\begin{prop}[{\cite[3.7]{Brochard_Mezard_ConjDeSmit}}]
\label{prop:wtf_equiv_plat}
 Let $R$ be an Artin local ring. Assume that $R$ is Gorenstein and contains a 
field. Then an $R$-module $M$ is flat over $R$ if and only if it is weakly 
torsion-free.
\end{prop}

\begin{proof}[{Proof of Theorem~\ref{thm_desmit}}]
 Let $M$ be a nonzero $A$-flat $B$-module. We will apply Lemma~\ref{lem:wtf} 
with $(x_1, \dots, x_n)$ the image in $B$ of a minimal system of generators of 
$\mgo_A$, and $(u_1, \dots, u_n)$ a system of generators of~$\mgo_B$, so that 
$J_x=\mgo_AB$ and $J_u=\mgo_B$. The assumption (i) holds because the 
morphism~$A\to B$ is local, and the assumption (ii) holds by 
\cite[4.2 and 4.3]{Brochard_Mezard_ConjDeSmit} because $M$ is flat over~$A$. 
Then $\Delta \notin \mgo_A B$, 
otherwise we would have $\Delta M\subset \mgo_A M$, which would imply 
that $M=\mgo_BM$ by Lemma~\ref{lem:wtf}, hence $M=0$, which is a contradiction. 
In particular, for any choice of the matrix $W$ and the generators~$u_1, \dots, 
u_n$ of~$\mgo_B$ such that $x=Wu$, we have $\det W \neq 0$. This proves that 
$\edim(B)=\edim(A)$, because if $\edim(B)<\edim(A)$ we can choose $W$ with a 
zero column. Moreover, in the ring $B/\mgo_A B$, the matrix $\ov{W}$ 
satisfies $\ov{W}u=0$ and 
$\det \ov{W}\neq 0$. In the language of~\cite[Definition 
2.6]{Simon_Strooker_Wiebe},
this means that $\ov{W}$ is (the transpose of) a $u$-Wiebe matrix for the ring 
$B/\mgo_AB$. By~\cite[2.7]{Simon_Strooker_Wiebe}, this implies that the ring 
$B/\mgo_AB$ is a complete intersection of dimension zero, and proves (2).

Let us prove that $M$ is $B$-flat. By~\cite[2.3]{Brochard_Mezard_ConjDeSmit}, 
it suffices to prove that $M/\mgo_AM$ is flat over $B/\mgo_AB$.
We have seen that~$B/\mgo_AB$ is a complete intersection of dimension zero and 
hence a Gorenstein Artin local ring. Moreover, it contains the 
field~$\kappa(A)$. Hence, by Proposition~\ref{prop:wtf_equiv_plat}, it
suffices to prove that $M/\mgo_AM$ is weakly torsion-free over $B/\mgo_AB$, 
i.e. that for any 
$\lambda\in B$ and any $m\in M$, the relation $\lambda m\in \mgo_AM$ 
implies that $\lambda\in \mgo_AB$ or $m\in \mgo_BM$. So, let $\lambda\in B$ and 
$m\in M$ be such that $\lambda m\in \mgo_AM$ and $\lambda\notin \mgo_AB$.
Let us prove that $m\in \mgo_BM$. By~\cite[2.7]{Simon_Strooker_Wiebe}, we know 
that the image $\ov{\Delta}$ of~$\Delta$ is a generator of the socle of the 
Gorenstein ring $B/\mgo_AB$. Hence, $(\ov{\Delta}) \subset (\ov{\lambda})$ in 
$B/\mgo_AB$. Since $\lambda m\in \mgo_A M$, this implies that $\Delta 
m\in\mgo_AM$. By Lemma~\ref{lem:wtf}, we get that $m\in \mgo_BM$, as 
required.

Lastly let us prove that $A\to B$ is flat. Let $E\to F$ be an injection of 
$A$-modules. Let $K$ 
be the kernel of $E\otimes_A B\to F\otimes_AB$. Since $M$ is $B$-flat, 
$K\otimes_B M$ is the kernel of $E\otimes_AM\to F\otimes_AM$, which is zero 
because $M$ is $A$-flat. But $M$ is faithfully flat over $B$ (because it is 
free and nonzero) and hence $K=0$.
\end{proof}

\begin{remarque}\rm
If we replace the assumption (iii) of Theorem~\ref{thm_desmit} with ``$A$ and 
$B$ are 
Artin or $M$ is of finite type over $A$'', then we do not need to assume that 
the morphism $\varphi : A\to B$ is local: it is then a consequence of the other 
hypotheses. Indeed, if there is an element $a\in \mgo_A$ such that 
$\varphi(a)\in B^{\times}$, then $M=\varphi(a)M$ and hence~$M=\mgo_A M$. Then 
$M=0$ by Nakayama's lemma (or because $\mgo_A$ is nilpotent if $A$ 
is Artin), which is a con\-tra\-dic\-tion. On the other hand, if $A$ is Artin, 
I am not sure that we really need the assumption that $B$ is Artin too.
\end{remarque}
\details{
Si on ne suppose pas $A\to B$ local, il ne suffit pas de supposer que $M$ est 
de type fini sur $B$. En effet, on a un contre-exemple en prenant $\Z_p\to 
\Q_p$ et $M=\Q_p$. Le morphisme n'est pas local, les dimensions de plongement 
ne sont pas égales.
}

The following is a global version of Theorem~\ref{thm_desmit}.

\begin{thm}
\label{thm:global_version}
 Let $f : X\to S$ be a morphism of locally Noetherian schemes. Assume that for 
any $x\in X$, the dimensions of the tangent spaces $T_x$ and $T_{f(x)}$ satisfy
\[
 \dim T_x \leq \dim T_{f(x)}
\]
Let $\Mc$ be a coherent $\Oc_X$-module, flat over $S$. Then:
\begin{enumerate}[label=(\arabic*)]
 \item $\Mc$ is flat over $X$.
\item For each point $x\in \Supp(\Mc)$, $f$ is flat at $x$, the fiber 
$f^{-1}(f(x))$ is a complete intersection of dimension zero at $x$ (i.e. its 
local ring at $x$ is a complete intersection of dimension zero) and $\dim 
T_x=\dim T_{f(x)}$.
\details{i.e. l'anneau local de la fibre en $x$ est une IC, cf déf EGA IV 
19.3.1.}%
\end{enumerate}
In particular, if there exists a coherent $\Oc_X$-module $\Mc$ flat over $S$ 
whose support is $X$, then $f$ is a complete intersection morphism of relative 
dimension zero (i.e. $f$ is flat and its fibers are zero-dimensional complete 
intersections).
\end{thm}
\begin{proof}
 Note that the local ring of $X$ at a point $x$ is unaltered through the base 
change along $\Spec\Oc_{S,f(x)} \to S$. By \cite[Ch. II, \S3, no 4, Prop. 
14]{Bourbaki_Algebre_Commutative1-4}, the module $\Mc_x$ is flat over 
$\Oc_{S,f(x)}$. Since $\dim T_x=\edim(\Oc_{X,x})$ and 
$\dim(T_{f(x)})=\edim(\Oc_{S, f(x)})$, we can apply Theorem~\ref{thm_desmit} 
and 
$\Mc_x$ is flat over $\Oc_{X,x}$. This proves~(1). If $x\in \Supp(\Mc)$ 
then~$\Mc_x\neq 0$ and by Theorem~\ref{thm_desmit} we deduce 
that $\Oc_{X,x}$ is flat over $\Oc_{S, f(x)}$, that these rings have the same 
embedding dimension and that 
$\Oc_{X,x}/\mgo\Oc_{X,x}$ is a zero-dimensional complete intersection (where 
$\mgo$ is the maximal ideal of~$\Oc_{S, 
f(x)}$). Since the latter ring $\Oc_{X,x}/\mgo\Oc_{X,x}$ identifies with the 
local ring at $x$ of the fiber $X_{f(x)}$, this proves (2).
\end{proof}

 \section{Application to the proof of Fermat's Last Theorem}
\label{section:preuve_flt}

To conclude his proof, Wiles had to prove that a certain morphism of 
$\Oc$-algebras $\Phi_{\Sigma} : R_{\Sigma} \to T_{\Sigma}$ is an isomorphism 
and that $T_{\Sigma}$ is a complete intersection over~$\Oc$. Here $\Oc$ is the 
ring of integers of a finite extension of $\Q_{\ell}$, $R_{\Sigma}$ is the 
universal deformation ring for a Galois representation $\ov{\rho} : 
\Gal(\ov{\Q}/\Q)\to \GL_2(k)$ (where $k$ is the residue field of $\Oc$) and 
$T_{\Sigma}$ is a certain Hecke algebra. The proof proceeds in two steps: first 
it is proved for $\Sigma=\vide$ (the ``minimal case''), then the general case is 
deduced from this one using some numerical criterion for complete intersections. 
To handle the minimal case, Wiles constructs a system of sets $Q_n$, 
$n\geq 1$, where $Q_n$ is a set of primes congruent to 1 mod $\ell^n$ with some 
other ad hoc properties. Then, by considering a kind of ``patching'' of the 
morphisms $\Phi_{Q_n}  : R_{Q_n} \to T_{Q_n}$, Wiles manages to prove the 
desired result for $\Sigma=\vide$. Among other things, the proof relied on the 
fact (due to Mazur and Ribet) that the homology of the modular curve is a free 
module of rank 2 over $T_{\vide}$ (known as a ``multiplicity one'' result). 
Diamond improved the method in~\cite{Diamond_The_Taylor_Wiles}: by patching the 
modules as well as the algebras, he managed to remove multiplicity one as an 
ingredient of the proof of Fermat's Last Theorem. His proof relies on his 
freeness criterion~\cite[2.1]{Diamond_The_Taylor_Wiles}. Since we only want 
to illustrate how our Theorem~\ref{thm_desmit} can be used to avoid the 
patching method and the use of Taylor-Wiles systems, we will work in the rather 
restrictive setting of~\cite{Darmon_Diamond_Taylor} 
and~\cite{Diamond_The_Taylor_Wiles}. We use their notations and statements in 
the sequel.

Let us prove that $\Phi_{\vide}$ is an isomorphism. Let $\lambda$ be a 
uniformizer of $\Oc$. By~\cite[2.49]{Darmon_Diamond_Taylor} (note that we only 
use it for $n=1$, so here also the proof can be simplified), there 
exists a finite set of prime numbers $Q$ such that:
\begin{itemize}
 \item $\edim(\ov{R_{Q}})\leq \# Q =:r$, \textrm{ where } 
$\ov{R_{Q}}=R_{Q}\otimes_{\Oc}k=R_Q/\lambda R_Q$
\item For any $q\in Q$, $q\equiv 1\ (\ell)$, $\ov{\rho}$ is unramified at $q$ 
and $\ov{\rho}(\Frob_q)$ has distinct eigenvalues.
\end{itemize}
Let $G$ be the $\ell$-Sylow subgroup of $\prod_{q\in Q} (\Z/q\Z)^{\times}$. 
We endow $R_Q$ with the structure 
of an $\Oc[G]$-algebra as in~\cite[\S 2.8]{Darmon_Diamond_Taylor}. Let $N$ 
be the integer $N_{\vide}$ of \cite[(4.2.1)]{Darmon_Diamond_Taylor} and 
$M=p^2\prod_{q\in Q}q$, where $p$ is a well-chosen auxiliary prime 
(see~\cite[\S4.3]{Darmon_Diamond_Taylor}). Consider the group
\[
 \Gamma=\Gamma_0(N)\cap \Gamma_1(M)
\]
and let $X_{\Gamma}$ be the associated modular curve. 
By~\cite[4.10]{Darmon_Diamond_Taylor}, there is an isomorphism $T_Q \to 
T'_{\mgo}$, where~$T'$ is the algebra of Hecke operators acting on 
$H^1(X_{\Gamma},\Oc)$ and $\mgo$ is a certain maximal ideal of~$T'$. Consider 
the $T'_{\mgo}$-module
\[
 H:=H^1(X_{\Gamma},\Oc)_{\mgo}^-
\]
of elements on which complex conjugation acts by -1. We can view $H$ as a 
module over $\Oc[G]$, $R_Q$, or $T_Q$ via the $\Oc$-algebra morphisms
\[
 \Oc[G]\to R_Q \to T_Q \to T'_{\mgo}\, .
\]
\begin{thm}
\label{thm:isom_ic}
 With the above notation:
\begin{enumerate}
 \item $H$ is free over $T_Q$,
\item $T_Q$ is free over $\Oc[G]$,
\item $T_Q$ is a relative complete intersection over $\Oc[G]$ (hence also over 
$\Oc$),
\item $\Phi_Q : R_Q \to T_Q$ is an isomorphism.
\end{enumerate}
\end{thm}
\begin{proof}
 By~\cite[Lemma~3.2]{Diamond_The_Taylor_Wiles}, $\ov{H}=H/\lambda H$ is free 
over $k[G]$. Note that there is an isomorphism $\Oc[G]\simeq \frac{\Oc[[S_1, 
\dots, S_r]]}{((1+S_1)^{\alpha_1}-1, \dots,(1+S_r)^{\alpha_r}-1)}$, where the 
$\alpha_i$ are the cardinalities of the $l$-Sylow subgroups of the 
$(\Z/q\Z)^{\times}$ for $q\in Q$. In particular, $\alpha_i\geq \ell\geq 2$ and 
it follows that $\edim(k[G])=r$. We can apply Theorem~\ref{thm_desmit} to the 
morphism $k[G]\to \ov{R_Q}$ and the module $\ov{H}$. We get that $\ov{H}$ is 
free over~$\ov{R_Q}$, and that $\ov{R_Q}$ is free and is a relative complete 
intersection over $k[G]$. In particular, $\ov{R_Q}\to \ov{T_Q}$ must be 
injective (otherwise $\ov{H}$ would have torsion over $\ov{R_Q}$). Since 
it is already known to be surjective, it is an isomorphism. We have proved all 
the statements after $\otimes_{\Oc}k$. Since $T_Q$ is free over $\Oc$, the 
theorem now follows from Lemma~\ref{rem:appli_nakayama} below.
\end{proof}

\begin{lem}
\label{rem:appli_nakayama} Let $\Oc$ be a Noetherian 
local ring with maximal ideal $\mgo$ and residue field $k$.
\begin{enumerate}
 \item Let $f : M\to N$ be a morphism of finite-type $\Oc$-modules. Assume that 
$N$ is free over $\Oc$ and that $f\otimes_{\Oc} k$ is an isomorphism. Then $f$ 
is an isomorphism.
\item Let $A$ be a finite $\Oc$-algebra and $M$ be an $A$-module which is free 
of finite rank over $\Oc$. If $M\otimes_{\Oc}k$ is free over $A\otimes_{\Oc}k$, 
then $M$ is free over $A$.
\item Assume that $\Oc$ is complete. Let $A\to B$ be a morphism of finite local 
$\Oc$-algebras, with $B$ free over $\Oc$. If $B\otimes_{\Oc}k$ is a relative 
complete intersection over $A\otimes_{\Oc}k$, then $B$ is a relative complete 
intersection over $A$ (i.e. the ring $B/\mgo_A B$ is a complete intersection).
\end{enumerate}
\end{lem}
\begin{proof}
 These are standard consequences of Nakayama's Lemma.
\end{proof}

\begin{remarque}\rm
 By~\cite[Cor. 2.45 and 3.32]{Darmon_Diamond_Taylor}, the canonical 
morphisms $R_Q \to R_{\vide}$ and $T_Q\to T_{\vide}$ induce isomorphisms 
$R_Q/\ago_QR_Q \simeq R_{\vide}$ and $T_Q/\ago_QT_Q \simeq T_{\vide}$, where 
$\ago_Q$ is the augmentation ideal of $\Oc[G]$ (i.e. the ideal generated by 
$S_1, \dots, S_R$). Hence, from Theorem~\ref{thm:isom_ic}, it follows 
immediately that 
$\Phi_{\vide} : R_{\vide} \to T_{\vide}$ is an isomorphism and $T_{\vide}$ 
is a complete intersection over $\Oc$.
\end{remarque}

\subsection*{Acknowledgments.} I warmly thank Bart de Smit for raising his 
conjecture and for pointing out to me the freeness criterion~\cite[Thm 
2.1]{Diamond_The_Taylor_Wiles}. I also thank Michel Raynaud and Stefano Morra 
for interesting conversations, and the anonymous referee for valuable comments 
that helped to clarify the exposition.

\bibliographystyle{plain}
\bibliography{mabiblio}
\end{document}